\documentclass[a4paper,11pt]{article}
\usepackage{amssymb,amsmath,amsthm}
\usepackage{marvosym}
\usepackage{graphicx}

\usepackage{todonotes, comment} 

\usepackage{amsfonts}
\usepackage{latexsym}
\usepackage{color}
\usepackage[english]{babel}
\usepackage{hyperref}
\usepackage{tcolorbox}
\usepackage{float}

\numberwithin{equation}{section}

\newtheorem{lemma}{Lemma}[section]
\newtheorem{theorem}[lemma]{Theorem}
\newtheorem{proposition}[lemma]{Proposition}
\newtheorem{corollary}[lemma]{Corollary}
\newtheorem{remark}[lemma]{Remark}

\newtheorem{definition}[lemma]{Definition}

\newcommand{\tb}{{\tt b}}

\def\tq{{\tt q}}
\def\td{{\tt d}}
\def\tb{{\tt b}}

\newcommand{\Tr}{\operatorname{Tr}}

\newcommand{\dd}{\,\mathrm d}
\newcommand{\braket}[2]{\langle#1|#2\rangle}

\providecommand{\Rk}{R_\kappa}
\providecommand{\Bt}{B_{\Delta}\null}

\providecommand{\Xu}{X_u}
\providecommand{\Xv}{X_v}

\newcommand{\uu}[1]{u^{(#1)}}
\newcommand{\vv}[1]{v^{(#1)}}

\newcommand{\R}{\mathbb R}
\newcommand{\C}{\mathbb C}

\newcommand{\N}{\mathbb N}

\newcommand{\Id}{I}

\definecolor{awesome}{rgb}{1.0, 0.13, 0.32}
\definecolor{darkgr}{rgb}{0.0, 0.62, 0.42}
\definecolor{cyan}{rgb}{0.0, 0.72, 0.92}

\def\cM{\mathcal{M}}

\def\cH{{\mathcal{H}}}

\def\QS{{\mathcal{Q}\kern-0.3pt\mathcal{S}}}

\def\cM{{\mathcal{M}}}
\def\cL{\mathcal L}

\def\cD{{\mathcal{D}}}

\def\cU{{\mathcal{U}}}

\def\tb{{\tt b}}
\def\ti{{\tt i}}
\def\tti{{\tilde{\tt \imath}}}
\def\Fdue{\null\,\null_2\kern-1pt F_{1}\null}

\begin{document}

	\title{Lax Pair for the Kirchhoff-Pohozaev equation
	}

	\date{}
	
	
	\author{ Dario Bambusi \footnote{Dipartimento di Matematica ``Federigo Enriques'', Universit\`a degli Studi di Milano, Via Saldini 50, 20133
Milano, Italy 
 \textit{Email: } \texttt{dario.bambusi@unimi.it}},
 Emanuele Haus \footnote{Dipartimento di Matematica e Fisica, Universit\`a degli Studi Roma Tre, Largo San Leonardo Murialdo 1, 00146
Roma, Italy 
 \textit{Email: } \texttt{emanuele.haus@uniroma3.it}},
 Simone Marrocco \footnote{Dipartimento di Matematica e Fisica, Universit\`a degli Studi Roma Tre, Largo San Leonardo Murialdo 1, 00146
Roma, Italy 
 \textit{Email: } \texttt{simone.marrocco@uniroma3.it}},
 Michela Procesi \footnote{Dipartimento di Matematica e Fisica, Universit\`a degli Studi Roma Tre, Largo San Leonardo Murialdo 1, 00146
Roma, Italy 
 \textit{Email: } \texttt{michela.procesi@uniroma3.it}}}

	\maketitle
	
	\begin{abstract}
We prove that the Kirchhoff-Pohozaev equation on any compact connected
Riemannian manifold without boundary admits a Lax pair $(L,B)$, where
$L$ is a self-adjoint finite-rank perturbation of the Laplace-Beltrami
operator and $B$ is a skew-adjoint operator of finite rank, both
depending analytically on the phase space variables. We study the
spectrum of $L$ using the perturbative determinant which is known
to determine its spectrum and to admit a product representation
similar to that of the generating function in KdV. We use this to show
that the dynamics of the integrals of motion constructed in
\cite{BoitiManfrin2026,BHMP,campos} admits a Lax pair formulation with the same
operator $L$ and to get a new proof of the involutivity of these
integrals.
                 	\end{abstract}
	
	\noindent
	
	\medskip

	\def\du{u^{(1)}\null}
	\def\dv{v^{(1)}\null}
	
	\def\qdue{Q_2}
	
	\newcommand{\ip}[2]{\left\langle #1,#2\right\rangle}
	\newcommand{\ketbra}[2]{\lvert #1\rangle\langle #2\rvert}
	\newcommand{\norm}[1]{\left\lVert #1\right\rVert}
	\newcommand{\T}{\mathbb T}
	
	\def\duu{u^{(2)}}
	\def\dvv{v^{(2)}}
	\def\duuu{u^{(3)}}
	\def\dvuu{v^{(3)}}
	
	\newcommand{\gscalar}{{\mathtt q}_2}
	\newcommand{\escalar}{\mathtt d_2}
	\newcommand{\tscalar}{\mathtt{q}_3}
	\newcommand{\sscalar}{\beta}
	\def\rscalar{\mathtt d_1}
	\def\bscalar{\mathtt b_1}
	
	\def\bsym#1#2{#1\vee #2 }
	\def\anti#1#2{#1\wedge #2}
	\def\dis{\Delta}

	\newcommand{\ket}[1]{\lvert #1\rangle}
	\newcommand{\bra}[1]{\langle #1\rvert}
\def\dg{\Delta_g}

	\section{Introduction}
	Let $\cM$ be a $C^\infty$ compact connected manifold and denote by $\cL$ the real
	Hilbert space $\mathcal L:= L^2_0(\cM,\R)$, of the $L^2$ functions on
	$\cM$ with zero average, endowed by the scalar product
	\begin{equation}
		\label{prodottoHermitiano}
		\ip{f}{g}:=  \int_{\cM}  f g .
\end{equation}
Denote by $\Delta_g$ the (negative) Laplace-Beltrami operator on $\cM$
and consider the Kirchhoff-Pohozaev equation (KiPo in the following)
\begin{equation}
  \label{KiPo}
u_{tt}-\frac{\dg u}{\left(1-\langle u,\Delta_g u\rangle\right)^2}=0\ , \qquad (t,x) \in \R\times\cM, 
\end{equation}
which is Hamiltonian with respect to the standard Darboux symplectic form on $\cL\times\cL$,  with Hamiltonian function 
\begin{equation}\label{hamiltonian1int}
	H(u,v):=\frac{1}{2}\langle v,v\rangle-\frac{1}{2}
      \frac{1}{1-\langle u,\Delta_gu\rangle},
\end{equation}
where $v:=u_t$ is the variable conjugated to $u$.

The Kirchhoff-Pohozaev equation \eqref{KiPo} was first introduced by
Pohozaev \cite{P2} as an example of a Kirchhoff-type equation with a
higher order conservation law, which also allowed him to prove the
global well-posedness of the equation.  Very recently, Boiti and
Manfrin (see \cite{BoitiManfrin2025}, \cite{BoitiManfrin2026}) proved
that equation \eqref{KiPo} admits infinitely many conservation laws,
at every order. Starting from this remarkable result, the papers
\cite{campos} and \cite{BHMP} provided an explicit construction of a
whole hierarchy of constants of motion, which in the following we shall denote by $H_k$,  which furthermore pairwise
Poisson commute. In \cite{campos} such integrals
of motion were also used in order to produce a priori estimates for
the Sobolev norms of the solutions of the KiPo equation and to prove
that its solutions have a remarkable property, namely that they are
equicontinuous. So, they play the same role as the higher order
Hamiltonians in standard integrable hierarchies.

In \cite{BHMP} we showed that if $\cM$ is the
the $1d$ torus, then the restriction of KiPo to any finite dimensional
invariant subspace (corresponding to data with only a finite number of
Fourier modes different from zero) is integrable.

In the present paper we  prove that on any compact
boundaryless smooth Riemannian manifold
\begin{center}
	{\it all the  Hamiltonians $H_k$ of the }{ KiPo} {\it 
hierarchy admit a Lax pair $(L,B^{(k)})$, with $L$ a self-adjoint rank 3
perturbation of the Laplace-Beltrami operator which is independent of
$k$, while the $B^{(k)}$'s  are  skew-symmetric operators. }
\end{center}

Thus, the spectrum of the
operator $L$ is invariant under the dynamics of all the Hamiltonian
flows of the KiPo hierarchy. 

\vskip10pt
We study the spectrum of $L$. First, through an explicit argument based on the form of $-\dg-L$ we
prove that 
	{\it any eigenvalue $\nu$ of $-\dg$, with multiplicity $n_\nu$,
is also an eigenvalue of $L$ with multiplicity  at least
$(n_\nu-2)_+$, namely the positive part of $n_\nu-2$. }

\smallskip

        The
consequence is that, if $n_\nu>2$ for at least one eigenvalue
$\nu$ of the Laplacian, then the eigenvalues of $L$ do not provide a sufficient number
of integrals of motion to ensure integrability of the KiPo
equation. Indeed, even restricting the KiPo equation to a
spectral subspace corresponding to one single eigenvalue $\nu$, then one gets
a Hamiltonian system with $n_\nu$ degrees of freedom, but the
nontrivial eigenvalues of $L$ only provide 2 nontrivial integrals of
motion, which are not enough to ensure integrability. {The situation is
clearly similar if more eigenspaces are excited. We emphasize that, as
discussed in \cite{BHMP},
also the set of the integrals of motion constructed in \cite{BoitiManfrin2026,campos,BHMP}
has the same cardinality as the relevant eigenvalues of the Lax
operator, so that they are not enough to ensure integrability of KiPo
on general manifolds, e.g. on $\T^d$ with $d\geq2$.}

Then we investigate the perturbative determinant \cite{kato,yafaev} of the Lax operator:
it is well known that the spectrum of a finite rank perturbation of an
operator with compact resolvent is described completely by the so
called perturbative determinant, a meromorphic function whose zeros are
the points of the spectrum of $L$. Furthermore, such a function admits
a product representation very similar to the product representation
enjoyed by the generating function of other integrable equations, like
KdV and NLS \cite{KAMeKDV,NLS}. So we compute explicitly such a function and we find that
it essentially coincides with the generating function constructed in
\cite{BHMP,campos}. 

As a byproduct of this result, we get that the generating function is
meromorphic, and is also a purely spectral object, so that one can use
it to obtain a {new proof of the involution of the other
  Hamiltonians of the
hierarchy, which makes the paper self-contained}. 

The proofs of our results proceed as follows. First we exhibit the Lax pair for the
KiPo equation and we verify that it actually enjoys the property of
being a Lax pair, then we compute the perturbative determinant that we
denote by $\dis(\kappa)=\dis(\kappa,u,v)$ and we
get an expression formally equivalent to that of the generating
function of \cite{BHMP}, and subsequently we show that for any value of $\kappa$
there exists a Lax pair for the flow of the Hamiltonian vector field
of $\dis(\kappa,u,v)$, with of course the same operator $L$ as for
the KiPo. Finally, we use this property to construct the operators $B^{(k)}$
which together with $L$ form the Lax pair for the Hamiltonian flow of
$H_k$. Finally we give the new proof of the involution of the $H_k$'s.

\medskip

\noindent{\it Declaration of generative AI use.}  During the
development of this work, the authors used OpenAI’s ChatGPT to explore
and refine some intermediate arguments. All mathematical statements
and proofs were independently verified by the authors, who take full
responsibility for the content.

\medskip

\noindent{\it Acknowledgments.}  DB was partially supported by
INdAM-GNFM. EH, SM and MP were partially supported by INdAM-GNAMPA.

\section{Main results}

\subsection{Preliminaries and Lax pair for the KiPo
  equation}\label{laxkipo} 
In this section we give the Lax pair for the KiPo equation. We
postpone the results on the whole hierarchy to Section \ref{hierarchy},
since they need more preliminaries.

Define the family of phase spaces contained in $\cL\times\cL$ of the
functions of class
	\begin{equation}\label{phase space}
		{\mathcal H}^s:= H^s(\cM,\R) \times H^{s-1}(\cM,\R)\,,\quad s\geq 1\,,
	\end{equation}
	endowed with the norm
	\begin{equation}\label{normuv}
		\|(u,v)\|_{\cH^s}^2:= \|u\|_{H^s}^2+  \|v\|_{H^{s-1}}^2\,,
	\end{equation}
        where
        $$
\|u\|_{H^s}^2:=\langle u,(-\Delta_g)^su\rangle\ .
        $$
We denote $D:= (-\Delta_g)^{\frac12}$ and set
	\begin{equation}
		u^{(k)}:=D^ku,\qquad v^{(k)}:=D^kv,\quad
		A:=\frac{1}{1+\norm{\du}^{2}},\quad
		\alpha:=\frac{A}{1+\sqrt A}.
		\label{eq:def-pq}
	\end{equation}
	We use the bra-ket notation denoting, for $g\in \mathcal L$,
	$\ket{g}\equiv g$ and $\bra{g} \in \mathcal L^*$ the linear continuous
	map $f\mapsto \ip{g}{f}$. 
	For vectors $f,g\in\cL$, we denote
	\begin{equation}
		\bsym{f}{g}:=\ketbra{f}{g}+\ketbra{g}{f},
		\qquad
		\anti{f}{g}:=\ketbra{f}{g}-\ketbra{g}{f}.
		\label{eq:products}
	\end{equation}
	Introduce, for $k\geq 2$, the real scalars 
	\begin{align}
		\label{(3)}
		\tq_k:=\ip{u^{(k-1)}}{v^{(k-1)}},\qquad
		\td_k:=\|u^{(k)}\|^2,\qquad
	\end{align}
	and the vector 
	\begin{equation}
		w:=\dv-\alpha\gscalar\du.
		\label{eq:w}
	\end{equation}
	We define two  operators
	\begin{equation}
		L:=D^{2}+M,
		\qquad
		B:=\alpha\anti{\dv}{\du},
		\label{eq:LB}
	\end{equation}
	where
	\begin{equation}
		M:=-\alpha\bsym{\du}{\duuu}
		+\frac{\alpha^{2}\escalar}{2}\bsym{\du}{\du}
		-\frac{A^{-1}}{2}\bsym{w}{w}.
		\label{eq:M}
	\end{equation}
    Note that $(u,v)\mapsto M(u,v)$ and $(u,v)\to B(u,v)$ are analytic
    maps from $\cH^3$ into the space of bounded linear operators on
    $\cL$. 

%
%
%
%
%

\begin{theorem}[\textbf{Lax Pair}]\label{Teomain}
Let $u\in H^3$. Then, the operators $(L,B)$ above form a
Lax pair for KiPo. More precisely if $t\mapsto u(t)$ is in
$C^2(\R,H^3)$ and $v(t):=u_t(t)$, define $L(t):=L(u(t),v(t))$ and
$B(t):=B(u(t),v(t))$. Then, if $u(t)$ fulfills KiPo \eqref{KiPo} one has
		\begin{equation}\label{lax0}
			\dot{L}=[B,L].
		\end{equation}
Vice versa, if $u(\cdot)\in C^2(\R,H^3)$ is such that $v:=u_t$ is not
identically zero, and \eqref{lax0} is fulfilled, then $u(t)$ fulfills
KiPo. 
	\end{theorem}

A standard consequence is that the spectrum of $L$ (which in our case
is pure point) is invariant under the dynamics of KiPo, so our next
step consists in studying the spectrum of $L$. 
Denote by $\sigma(-\Delta_g)$ the spectrum of $-\dg$. For $\nu\in\sigma(-\Delta_g)$ denote by $V_\nu\subset\cL$
the eigenspace of $-\dg$ corresponding to the eigenvalue $\nu$ and by $n_\nu$ its
dimension, namely the multiplicity of $\nu$. 

It is well known that, since $L$ is a rank 3 perturbation of $-\dg$,
each eigenvalue  $\nu\in\sigma(-\dg)$ is also an eigenvalue of $L$
with multiplicity at least $n_\nu-3$. Actually, we have the following
proposition which will be proved at the end of Sect. \ref{Lax},

\begin{proposition}
  \label{solodue}
Let $\nu\in\sigma(-\dg)$, then $\nu$ is also an
eigenvalue of $L$ with multiplicity at least $(n_\nu-2)_+$, the
positive part of $n_\nu-2$. 
\end{proposition}

Thus, at most 2 eigenvalues leave $\sigma(-\dg)$ for each
$\nu$.  This makes apparent that {if $n_\nu\geq 3$ for at least
  one value of $\nu$, then} {\it the eigenvalues of $L$ cannot form a
complete set of integrals of motion for KiPo}.  Indeed $n_\nu$ is the
number of degrees of freedom corresponding to a datum $(u,v)\in
V_\nu\times V_\nu$, {while the independent integrals of motion are at
most 2}. We leave the investigation of the possible actual
integrability of KiPo on general manifold to future work.

\subsection{Generating function, KiPo
  Hierarchy and perturbative determinant}\label{hierarchy} 

In this subsection we recall a few facts on the KiPo hierarchy. First
of all we need to introduce some further notation, for $k\geq 2$:
\begin{equation}\label{BDQ}
	\begin{aligned}
			\tb_k:=\|v^{(k-1)}\|^2\quad\quad&
			E_k:=\dfrac{\tb_k}{A}+A \td_k\quad\quad&
		F_k:=\dfrac{\tb_k}{A}-A \td_k.
	\end{aligned}
\end{equation}
It was proved in \cite{campos,BHMP} that, for
$k\geq2$, the following quantities are integral of motion in
involution for the KiPo
flow
	\begin{equation}\label{def I}
		H_k=\frac12 ( E_k- \sum_{j=2}^k \tq_j \tq_{k+2-j} + \frac14
                \sum_{j=2}^{k-1}(E_j E_{k+1-j}-F_j F_{k+1-j})). 
	\end{equation}
the proof is recursive
and immediately extends to the case of general manifolds. {Here we
  do not use this property, instead we give a new proof of it.}

We remark that explicitly one has
\begin{equation}
  \label{esplicito}
H_k={\frac{1}{2}}\int_{\cM}\left(\frac{1}{A}v(x)(-\dg)^{k-1}v(x)+A
u(x)(-\dg)^{k}u(x)
\right)dx +l.o.t.
  \end{equation}
where l.o.t.\ means terms containing less derivatives, so that the
corresponding Hamilton equation has the form
\begin{equation}
  \label{KiPok}
u_{tt}-\dg^{2k-1}u+l.o.t.=0.
\end{equation}
%
%

\medskip
The next result gives the Lax pair for the Hamiltonian flow of $H_k$. 

\begin{theorem}\label{BKAPPA}
  For any $k\geq 2$, define the skew-adjoint
  operators
  	\begin{equation}\label{Bm}
  	\begin{aligned}
  		B^{(k)}(u,v)&=\alpha  \td_{k}  \, u^{(1)}\wedge v^{(1)}\\
  		&-\frac{1}{1+\sqrt A}\,u^{(1)}\wedge v^{(2k-1)}-\sum_{i=1}^{k-2}u^{(2i+1)}\wedge v^{(2(k-i)-1)}\,,
  	\end{aligned}
  \end{equation}
Then, if $t\mapsto (u(t),\,v(t))$ is a sufficiently smooth solution of
the equations of motion of $H_k$ and $L(t):=L(u(t),v(t))$,
$B^{(k)}(t):=B^{(k)}(u(t),v(t))$, one has
	\begin{equation}\label{bikappa}
		\dot{L}=[B^{(k)},L]\,.
	\end{equation}
\end{theorem}

Finally  in \cite{BHMP} we defined the formal power series
\begin{equation}
	\label{generetrice}
	J(\lambda):=2\sum_{k\geq 2} \lambda^{k-1} H_k\ ,
\end{equation}
which was used as a tool for {constructing the constants
  of motion with negative index $k$}.
 Actually a slightly more accurate analysis shows that there exists a
 function, which we still denote by $J$, that is well defined for
 $(u,v)\in \cH^2$, is analytic  for $\lambda$ in a domain containing
$\C\setminus[0,1]$, and coincides with $J$ at the level of formal power series.
Notably, the KiPo Hamiltonian
\eqref{hamiltonian1int} is recovered as
$$
H=-\frac{1}{2}\left(\lim_{\lambda\to\infty}J(\lambda)+1\right).
$$
{As shown in Section 4.1 of \cite{BHMP} the generating function
  can also be expressed through the functions\footnote{which however in \cite{BHMP}
  were denoted with the names $\tilde D$, $\tilde B$, $\tilde Q$}
\begin{align}
	\label{varie}
	U(\kappa)&:=A^{-1}-\langle u^{(2)};R_\kappa u^{(2)} \rangle\ ,\quad
	V(\kappa):=A-\langle \dv;R_\kappa\dv\rangle\ ,\quad
	\\
	Q(\kappa)&:=\langle \du;R_\kappa\dv\rangle\ .\quad 
\end{align}
 Indeed one has
	\[
	U(\kappa)V(\kappa)-\kappa Q^2(\kappa)= {J}(\kappa^{-1})+1.
	\]
}

As anticipated in the introduction, it  turns out that the
generating function essentially coincide with the perturbative
determinant of $L$ (see Appendix \ref{kato} for a review of the
definition and the properties of the pertrubative determinant).  To
state the result in a precise way, consider the resolvent of the
Laplacian, namely $R_\kappa:=(-\dg-\kappa)^{-1}$ and define the
perturbative determinant\footnote{as recalled in the appendix, the
determinant is defined as the determinant of the matrix obtained by
restricting $I+MR_\kappa$ to any subspace containing the range of $M$
(see \cite{kato}).  }
\begin{equation}
  \label{Pert.det}
\dis(\kappa):=\det(I+MR_\kappa).
\end{equation}

Denote by
$$
\lambda_1\leq\lambda_2\leq...\leq \lambda_n\leq...
$$
the sequence of the eigenvalues of $L$ and by 
$$
\nu_1\leq\nu_2\leq...\leq \nu_n\leq...
$$
the sequence of the eigenvalues of $-\dg$. The following theorem holds
\begin{theorem}
  \label{perturbative}
Consider the perturbative determinant \eqref{Pert.det} of the KiPo Lax
operator, then one has
\begin{equation}
  \label{funz.gen}
\dis(\kappa)={U(\kappa)V(\kappa)-\kappa Q^2(\kappa)=}1+{ J}\left(\frac{1}{\kappa}\right),
\end{equation}
where ${J}\left(\lambda\right)$ is the generating
function \eqref{generetrice}. Furthermore,
the following formula holds
\begin{equation}
  \label{prodotto}
\dis(\kappa)=\prod_{j\geq1}\frac{\lambda_j-\kappa}{\nu_j-\kappa}.
\end{equation}
\end{theorem}

We now introduce the sequence of the relevant eigenvalues of $L$, which
is useful in order to take into account the result of Proposition
\ref{solodue}. Indeed, the sequence of the eigenvalues of
$L$ contains also each eigenvalue $\nu$ of $-\dg$ with multiplicity at least
equal to the positive part
$(n_\nu-2)_+$ of $n_\nu-2$.

\begin{definition}
  \label{relevant}
Consider the sequence of the eigenvalues of $L$ repeated according to
their multiplicity and erase from it $(n_\nu-2)_+$ copies of each eigenvalue
$\nu\in\sigma(-\dg)$. We call the so obtained sequence {relevant part
  of the spectrum of $L$} and we denote it by $\sigma'(L)$. Its points
will be called {\it relevant eigenvalues} of $L$.
\end{definition}
We label the relevant eigenvalues $\lambda'_n$ in increasing order (with
repetitions), so that they fulfill
\begin{equation}
  \label{relevant.3}
\lambda'_1\leq\lambda'_2\leq...\leq\lambda'_n\leq... \ .
\end{equation}

We do a similar excision also from the spectrum of $-\dg$.
\begin{definition}
  \label{relevant.1}
Consider the sequence of the eigenvalues of $-\dg$ repeated according to
their multiplicity and erase from it $(n_\nu-2)_+$ copies of 
each $\nu\in\sigma(-\dg)$. We call the so obtained sequence {relevant part
  of the spectrum of $-\dg$} and we denote it by $\sigma'(-\dg)$. Its points
will be called {\it relevant eigenvalues} of $L$ and denoted by $\nu' $.
\end{definition}
As in Eq. \eqref{relevant.3} we label the relevant eigenvalues
$\nu'_n$ of
$-\dg$ in increasing order:
\begin{equation}
  \label{relevant.3.1}
\nu'_1\leq\nu'_2\leq...\leq\nu'_n\leq... \ .
\end{equation}

With these definitions we have the following
\begin{corollary}
One has
    \begin{equation}
\label{rilevanti}
      \dis(\kappa)=\prod_{n\geq1}\frac{\lambda_{n}'-\kappa}{\nu'_n-\kappa}.
		\end{equation}
\end{corollary}
Such a Corollary is relevant because it shows that $\dis$ is
completely determined by the sequence of the relevant eigenvalues, so
it is a spectral function.
{In Sect. \ref{simple} we will prove that for the states with
  finitely many spectral components, the relevant eigenvalues
are generically simple and Poisson commute. As a
corollary we will deduce the following theorem, which is our last
result.}

\begin{theorem}
  \label{ridimostro}
Consider the functions \eqref{def I}, then, forall $k_1,k_2\geq 2$ one
has
\begin{equation}
  \label{ridi}
\left\{H_{k_1},H_{k_2}\right\}=0\ ,\quad \forall
(u,v)\in\cH^{\max\left\{k_1,k_2\right\}} 
\end{equation}
and also
\begin{equation}
  \label{conH}
\left\{H,H_k\right\}=0\ ,\quad \forall
(u,v)\in\cH^{k}.
  \end{equation}
\end{theorem}

\section{Proof of the  Lax pair result}\label{Lax}

In this section we prove the result on the Lax Pair of the KiPo
equation. {In the Lemmas we will not specify the needed smoothness. We
will always assume $u(\cdot)\in C^2\left(\R,H^3\right)$.}

Let us define
\begin{align}
\label{esse}
S&:=  \Id-\alpha\ket{u^{(1)}}\bra{u^{(1)}}
\\
		L_1&:=D^2-A^{-1}\ketbra{\dv }{\dv }.
		\label{eq:definition-L11}
\end{align}
In order to prove Theorem \ref{Teomain} we will use the representation
	\begin{equation}
		\label{glg1}
		L=SL_1S\ ,
	\end{equation}
which follows from a direct simple computation,
so first of all we compute a couple of formulae for $\dot S$. 

	\begin{lemma}\label{lemmaboh1}
		We have
		\begin{equation}\label{eq:Pdot}
			\dot S-BS=-A\,S\,\ketbra{v^{(1)}}{u^{(1)}},
			\qquad
			\dot S+SB=-A\,\ketbra{u^{(1)}}{v^{(1)}}\,S .
		\end{equation}
		\end{lemma}
		\begin{proof}
			From the definition of $S$, we get
			\begin{equation*}
				\dot S=-\dot\alpha\,\ketbra{u^{(1)}}{u^{(1)}}
				-\alpha\bigl(\ketbra{v^{(1)}}{u^{(1)}}+\ketbra{u^{(1)}}{v^{(1)}}\bigr).
			\end{equation*}
			and
			\begin{equation*}
				BS=\alpha\bigl(1-\alpha\norm{u^{(1)}}^2\bigr)\ketbra{v^{(1)}}{u^{(1)}}
				-\alpha\,\ketbra{u^{(1)}}{v^{(1)}}
				+\alpha^2q_2\,\ketbra{u^{(1)}}{u^{(1)}}.
			\end{equation*}
				Subtracting, the terms $\ketbra{\uu1}{\vv1}$ cancel and
			\begin{equation}\label{eq:PdotBP}
				\dot S-BS=-\bigl(\dot\alpha+\alpha^2\mathtt{q}_2\bigr)\ketbra{\uu1}{\uu1}
				-\alpha\bigl(2-\alpha\norm{\uu1}^2\bigr)\ketbra{\vv1}{\uu1}.
			\end{equation}
			By \eqref{eq:def-pq},
			$\alpha\bigl(2-\alpha\norm{\uu1}^2\bigr)=A$.\\
			Note that, since
			$\frac{d}{dt}\norm{\uu1}^2=2\mathtt{q}_2$, we have $\dot\alpha=-\alpha(\alpha+A)\mathtt{q}_2$. 
			Inserting these two facts into \eqref{eq:PdotBP}, we obtain the first identity in \eqref{eq:Pdot}. The second one follows by
			taking adjoints.
		\end{proof}
Then we have the following Lemma 

\begin{lemma}\label{Laxgen}
		Let $v:=u_t$, then we have
		\begin{equation}\label{carina1}
			\dot{L}-[B,L]=-A^{-1}\Bigl[\left(SD(\dot{v}+A^2u^{(2)})\right)\vee(Sv^{(1)})\Bigr]
		\end{equation}
	\end{lemma}

                \begin{proof}
			From \eqref{glg1} , we have 
			\begin{equation*}
		\dot L-[B,L]=(\dot S-BS)\,L_1S+S\dot L_1S+SL_1\,(\dot S+SB),
					\end{equation*}
	that, together with Lemma \ref{lemmaboh1}, yields
		\begin{equation}\label{eq:conj}
			\dot L-[B,L]
			=S\Bigl(-A\,\ketbra{\vv1}{\uu1}\,L_1-A\,L_1\,\ketbra{\uu1}{\vv1}+\dot L_1\Bigr)S .
		\end{equation}
		From \eqref{eq:definition-L11} we have
		\begin{equation}
			\dot L_1=-2\mathtt{q}_2\,\ketbra{v^{(1)}}{v^{(1)}}-A^{-1}\,(D\dot v)\vee v^{(1)}
		\end{equation}
Finally, noting that $\bra{u^{(1)}}L_1=\bra{u^{(3)}}-A^{-1}\mathtt{q}_2\bra{v^{(1)}}$, we obtain
		\begin{align*}
			-A\,\ketbra{\vv1}{\uu1}\,L_1-A\,L_1\,\ketbra{\uu1}{\vv1}+\dot L_1
			&=-A\,u^{(3)}\vee v^{(1)}-A^{-1}\,(D\dot v)\vee v^{(1)}\\
			&=-A^{-1}\bigl(D(\dot v+A^2u^{(2)})\bigr)\vee v^{(1)}.
		\end{align*}
Substituting in the \eqref{eq:conj} we get the \eqref{carina1}.
	\end{proof}

	\begin{proof}[Proof of Theorem \ref{Teomain} (i) ]
		If we assume that $t\mapsto (u(t),v(t))$ is a sufficiently regular solution of \eqref{KiPo} then for all $t$ one has $\dot{v}+A^2u^{(2)}$ so that, by \eqref{carina1}, we deduce that \eqref{lax0} holds.
		
	In order to prove that \eqref{lax0} implies \eqref{KiPo} (provided we assume $\dot u=v\not\equiv0$), we use \eqref{carina1} and start by noticing that, at any fixed time $t$, the condition $\left(SD(\dot{v}+A^2u^{(2)})\right)\vee(Sv^{(1)})=0$ implies $SD(\dot{v}+A^2u^{(2)})=0$ or $Sv^{(1)}=0$. Recall that $S$ is invertible, and that $D$ is invertible on spaces of functions with zero average in $x$. Thus, equation \eqref{lax0} implies that, at any given time, either $v=0$ or $\dot v+A^2u^{(2)}=0$. Our goal is to prove that, indeed, one has $\dot v+A^2u^{(2)}=0$ for all $t$, i.e.\ \eqref{KiPo} is satisfied. Now, the case where $v(t)=0$ for all $t$ is excluded by the assumption that $\dot u=v$ and that $(u(t),v(t))$ is not constant in time. Since $u$ is $C^2$ in time by assumption, both $v$ and $\dot v+A^2u^{(2)}$ are continuous in time, so the only case left to exclude is that $v\equiv0$, with $\dot v+A^2u^{(2)}\not\equiv0$, on some time intervals of positive length and $\dot v+A^2u^{(2)}\equiv0$, with $v\not\equiv0$, on some other time intervals of positive length.
	
To rule this out, assume by contradiction that there exists an interval $I\subsetneqq\R$ of positive length such that: $(i)$ $\dot v+A^2u^{(2)}\equiv0$ on $I$; $(ii)$ there is no interval $J\supsetneqq I$ with $\dot v+A^2u^{(2)}\equiv0$ on $J$ (i.e.\ $I$ is maximal w.r.t.\ the property that $\dot v+A^2u^{(2)}\equiv0$); $(iii)$ $v\not\equiv0$ on $I$. Let $t_0$ be an endpoint of $I$. Since $u$ is $C^2$ in time, we have $\dot v(t_0)+A^2u^{(2)}(t_0)=0$. By the maximality of $I$, there exists a sequence $t_n \to t_0$ such that $\dot v(t_n)+A^2u^{(2)}(t_n)\neq0$. Therefore, one must have $v(t_n)=0$, which in turn implies $v(t_0)=0$ by continuity. Since $v$ is $C^1$ in time, there exists $\dot v(t_0)=\lim_{n\to\infty} (v(t_n)-v(t_0))/(t_n-t_0)=0$. Plugging this into $\dot v(t_0)+A^2u^{(2)}(t_0)=0$, we deduce $u^{(2)}(t_0)=0$. By the injectivity of the Laplacian on spaces of functions with zero average, we get $u(t_0)=0$. We have $u(t_0)=0$, $v(t_0)=0$ and, by the well-posedness of the Cauchy problem for \eqref{KiPo} on the interval $I$, this implies $u\equiv v\equiv 0$ on $I$, which is a contradiction.
\end{proof}

We end this section with the proof of Proposition \ref{solodue}.

\noindent{\it Proof of Proposition \ref{solodue}}
 Given $\nu\in \sigma(-\dg)$, we denote by $V_\nu$ the related
 eigenspace and by $\Pi_\nu$ the orthogonal projector on $V_\nu$.
We have that $\nu$ is an eigenvalue for $L$ if there is a solution of 
\begin{equation}\label{boh3}
(D^2-\nu)x=-Mx.
\end{equation}
It is clear that all the elements of $V_{\nu}\cap \ker(M)$ are
solutions of \eqref{boh3}. So we study the Kernel of the
restriction of $M$ to $V_\nu$. By Eq. \eqref{eq:M}, one has that $f\in
V_{\nu}\cap \ker(M)$ if and only if
\begin{align*}
  0=\ip{u^{(1)}}f=\ip{\Pi_\nu
    u^{(1)}}f=\nu^{1/2}\ip{u}f\ ,\quad\iff\quad \ip u f=0
  \\
  0=\ip{u^{(3)}}f=\ip{\Pi_\nu u^{(3)}}f=\nu^{3/2}\ip{u}f\ ,
  \quad\iff\quad \ip u f=0
  \\
  0=\ip{v^{(1)}}f=\ip{\Pi_\nu v^{(1)}}f=\nu^{1/2}\ip{v}f\ ,
  \quad\iff\quad \ip v f=0,
\end{align*}
so if and only if $\ip u f=\ip v f=0$, which shows that
$$
\dim\left(V_\nu\cap Ker(M)\right)\geq (n_\nu-2)_+,
$$
which immediately implies the thesis.
\qed

\section{The perturbative determinant and its product representation. }\label{birman} 

In this section we prove Theorem \ref{perturbative} and prove the 
explicit formula for the perturbative determinant. 


\begin{lemma}
The identity 
	\begin{equation}\label{identitaU}
		U(\kappa)=1-\kappa\langle u^{(1)};R_\kappa u^{(1)}\rangle
	\end{equation}
holds.
\end{lemma}

\begin{proof} One has
		\begin{equation}
		\begin{aligned}
			U(\kappa):&=A^{-1}-\langle u^{(2)};R_\kappa u^{(2)}\rangle =1+\|u^{(1)}\|^2-\langle u^{(2)};R_\kappa u^{(2)}\rangle\\
			&=1+\|u^{(1)}\|^2-\langle D^2u;(D^2-\kappa\Id)^{-1} D^2u\rangle\\
			&=1+\|u^{(1)}\|^2-\langle D^2u;(D^2-\kappa\Id)^{-1} (D^2-\kappa+\kappa)u\rangle\\
			&=1-\kappa\langle u^{(1)};R_\kappa u^{(1)}\rangle\,.
		\end{aligned}
	\end{equation}
\end{proof}

\begin{proposition}
	\label{generatrice.1}
Let $\dis$ be defined by \eqref{Pert.det}, then one has
	\begin{equation}
		\label{bir.1}
\dis(\kappa)=J(\kappa^{-1})+1\,.
	\end{equation}
       \end{proposition}
\begin{proof}
We compute the determinant on the l.h.s. of \eqref{Pert.det}; to this end,
it is useful to define $G:=S^{-2}$.
By construction (see \eqref{esse}) and
$$\alpha\bigl(2-\alpha\norm{\uu1}^2\bigr)=A
$$
we have
\[
S^2 = \Id - \alpha(2  -\alpha  \|\du\|^2) \ketbra{\du}{\du} = \Id - A \ketbra{\du}{\du} 
\] 
so that (using $A(1+\|\du\|^2)=1$)
\begin{equation*}
		G=I+\lvert \du\rangle\langle \du \rvert .
	\end{equation*}

Since \(SGS=\Id\), formula \eqref{glg1} gives
\begin{equation*}
	L-\kappa \Id 
	=SL_1S-\kappa SGS
	=S(L_1-\kappa G)S.
	\end{equation*}
It follows that
\begin{equation*}
	(L-\kappa\Id)R_\kappa
	=S\big[(L_1-\kappa G)R_\kappa\big]
	\big[R_\kappa^{-1}SR_\kappa\big].
	\end{equation*}
Each factor on the right-hand side is the identity plus a finite-rank
operator, more precisely
\begin{align*}
	R_\kappa^{-1}SR_\kappa
	=\Id-\alpha
	\ket{(D^2-\kappa)\du}
	\bra{\du}R_\kappa.
	\\
	(L_1-\kappa G)R_\kappa
	=\Id
	-A^{-1}\ketbra{\dv}{\dv}R_\kappa
	-\kappa\ketbra{\du}{\du}R_\kappa.
	\end{align*}
Therefore, one has
\begin{equation}
  \label{delta.P}
  \det(\Id+MR_\kappa)
	=\det(S)\,
	\det\big[(L_1-\kappa G)R_\kappa\big]\,
	\det\big[R_\kappa^{-1}SR_\kappa\big].
\end{equation}
One can compute each determinant using Eq. \eqref{A.1}
which gives
\begin{equation}
	\det(S)=\sqrt A,\qquad
	\det\big[R_\kappa^{-1}SR_\kappa\big]=\sqrt A\ ,
	\label{eq:det-S}
\end{equation}
while 
\begin{align}
	\label{detL_1}
	\det\big[(L_1-\kappa G)R_\kappa\big]
	& =  \det
	\begin{pmatrix}
		1-\kappa \ip{\du}{R_\kappa\du} & -\kappa Q(\kappa)\\
		-A^{-1}Q(\kappa)   & 1-A^{-1}\ip{\dv}{R_\kappa\dv}
	\end{pmatrix}
	\\
	&=\det
	\begin{pmatrix}
		U(\kappa) & -\kappa Q(\kappa)\\
		-A^{-1}Q(\kappa)   & A^{-1}V(\kappa)
	\end{pmatrix} .
\end{align}
Combining with \eqref{eq:det-S}, one gets the result.
\end{proof}

\section{On the simplicity of the relevant spectrum}\label{simple}

We prove here that for states with finitely many spectral components,
generically all the relevant eigenvalues $\lambda_n'$ are distinct and
lie outside $\sigma(-\dg)$.

  Let $S\subset \sigma(-\dg)$ be finite. Denote by
  $$
V_S:=\mathop{\oplus}_{\nu\in S}V_\nu
$$ the spectral subspace corresponding to $S$ and remark that if
\begin{equation}
  \label{sottospazio}
(u,v)\in V_S\times V_S
  \end{equation}
then $V_S$ is invariant under the action of $L$. Assume
\eqref{sottospazio} and denote 
 by
$L_S:=L\big|_{V_S}$ the restriction of $L$ to 
$V_S$,
{so that $L= L_S \oplus -\dg\big|_{V_S^\perp}$.}

We have the next Lemma.

\begin{lemma}
  \label{finite gap}
There exists an open and dense subset $\cU_S\subset V_S\oplus
V_S $ s.t., if $(u,v)\in \cU_S$ then $\sigma'(L_S)$ is simple {and the
relevant eigenvalues are analytic functions of $(u,v)$}.
\end{lemma}
\proof {The proof is based on the fact} that the relevant part of
the spectrum of $L_S$ is the solution of a polynomial equation with
coefficients that are analytic functions of $(u,v)\in V_S\oplus V_S$,
namely the ratio between the characteristic equation of $L_S$ and
$\prod_{\nu\in S}(\nu-\kappa)^{(n_\nu-2)_+}$. {The regularity of the
coefficients  is deduced from the regularity of the coefficients of the
characteristic equation of $L_S$ and the fact that $\prod_{\nu\in
  S}(\nu-\kappa)^{(n_\nu-2)_+}$ is independent of $u$ and $v$.}

{In particular one thus has that if the relevant eigenvalues are simple
then they are also analytic functions of $(u,v)$.}

From analyticity of the coefficients it also follows that the
discriminant\footnote{We recall that the discriminant of a polynomial
$P$ is a polynomial expression in the coefficients of $P$, with the
property that it vanishes if and only if $P$ has a repeated root} of
this polynomial is an analytic function of $(u,v)$. So, if we show
that such a discriminant is a nontrivial function, then, by
analyticity it is nonvanishing on an open and dense subset of
$V_S\oplus V_S$. 

To show that the above discriminant is a nontrivial function it is
enough to exhibit one couple $(u,v)$ with the property that the relevant
part of the spectrum of $L_S$ is simple. This is what we are now going
to do.

We consider only the case where $n_\nu\geq 2$ for all $\nu\in S$. The
other case can be easily obtained by simply eliminating the part of
the forthcoming construction involving $v_{\nu}$.
We start by decomposing $u,v$ in spectral components
\[
(u,v)=\{u_\nu,v_\nu\}_{\nu \in S}\,,\qquad (u_\nu,v_\nu) \in V_\nu\oplus V_\nu
\] 
{with $u_\nu$ orthogonal to $v_\nu$ for all $\nu\in S$ and}
with $\|u_\nu\|=1$ and $\|v_\nu\|= \tau$ (here $\tau$ is some parameter we modulate). Then one can easily verify that $F_1={\rm span} _{\nu\in S}u_\nu$ and $F_2:={\rm span}_{\nu\in S}{
 v_\nu}$  are invariant for $L$, while on the orthogonal $L$ acts
 as $-\dg$.
We construct the eigenvalues of $L$ on each one of these
subspaces and show that one of them for each subspace is the relevant
eigenvalue. 

First we remark that, with this choice, 
\begin{align}
  \label{}
L\big|_{F_1}=SD^2S\ ,\quad
L\big|_{F_2}=D^2-A^{-1}\ketbra{v^{(1)}}{v^{(1)}}. 
\end{align}
So the eigenvalues of $L\big|_{F_1}$ and of $L\big|_{F_2}$ not belonging to the spectrum of $-\dg$
are the zeros of the perturbative determinants of each one of
them. Now, by the computation in Appendix \ref{app:rank-one} one has
that the perturbative determinant of $L\big|_{F_2}$ is given by
\begin{equation}
  \label{f2}
\dis_{F_2}(\kappa)=1-A^{-1}\ip{\dv}{R_\kappa\dv}=1- A^{-1}\sum_{\nu\in
  S}\tau^2\frac{\nu}{\nu-\kappa}.
  \end{equation}

Remark that, since
$$
\frac{d\dis_{F_2}}{d\kappa}=-\sum_{\nu\in
  S}\tau^2\frac{\nu}{(\nu-\kappa)^2}<0,
$$
the function $\dis_{F_2}$ has exactly one zero between two different
values of $\nu$. As a consequence $L\big|_{F_2}$  has simple relevant spectrum.  

Concerning the perturbative determinant of $L\big|_{F_1}$, proceeding
as in Appendix \ref{app:rank-one}  we get that it is
given by
\begin{equation}
  \label{f1}
\dis_{F_1}(\kappa)=A\left(1+\kappa\ip{\du}{R_\kappa\du}\right)=A\left(1-\kappa \sum_{\nu\in S}\frac{\nu}{\nu-\kappa}\right),
\end{equation}
so that,  as before, $L\big|_{F_1}$  has simple relevant spectrum.

Now, let $\overline{\lambda'}$ be a solution of \eqref{f1}=0  then
there is at most one value of $\tau$, s.t. it
is also a solution of \eqref{f2}=0. 
So, excluding at most $|S|$  values
of $\tau$ one gets that each relevant eigenvalue ${\lambda'}$ has
multiplicity 1.

It follows that
for this choice of $u,v$ all the eigenvalues have multiplicty one and
therefore the discriminant is different from zero and therefore is a
nontrivial analytic function. 
\qed

\section{Lax pair for the generating function}\label{lax_pair_generating}
We consider the Hamiltonian {vector field} associated to the generating
function  $\dis(\kappa)= U(\kappa)V(\kappa)
-\kappa Q^2(\kappa)$, where the functions $U,V,Q$ are defined in
\eqref{varie} (see \eqref{bir.1}).

\begin{lemma}
  \label{Xu.1}
  The Hamilton equations of $\dis(u,v,\kappa)$ are given by
  \begin{equation}\label{Xu}
	\begin{cases}
		\dot{u}= X_u:=   -2 U(\kappa) R_\kappa v^{(2)} -2 \kappa Q(\kappa) R_\kappa u^{(2)} \\
		\dot v = X_v:= 2 \kappa V(\kappa) R_\kappa u^{(2)} + 2 U (\kappa) A^2 u^{(2)} +2 \kappa Q(\kappa) R_\kappa v^{(2)} 
	\end{cases}
\end{equation}
\end{lemma}
\begin{proof}
	We recall that the Hamiltonian vector field associated to the standard Darboux $2$-form is 
	\[
	X_u = \nabla_v \dis \,,\quad X_v = -\nabla_u \dis
	\]
	where $\nabla$ is the $L^2$ gradient.
	By direct computation
	\[
\begin{aligned}
		&\nabla_u  A \;\;\;\;\;=- 2 A^2 \duu \,,\qquad\quad	\nabla_v  A \, = 0\,,\\  &\nabla_u U(\kappa) =-2 \kappa  R_\kappa u^{(2)} \,, \quad \nabla_v U(\kappa) = 0
	\,,	\\
		&\nabla_u V(\kappa) = -2 A^2 \duu \,, \quad\;\;\,	\nabla_v V(\kappa) = -2 R_\kappa v^{(2)}\,, \\ &  \nabla_u Q(\kappa) =   R_\kappa v^{(2)} \,, \quad\qquad \nabla_v Q(\kappa) = R_\kappa u^{(2)}
\end{aligned}
	\]
		Then substituting in \eqref{bir.1} we obtain the result.
\end{proof}


Define $\beta=
    \dfrac{\sqrt{A}}{1+\sqrt{A}}$, we have the following proposition.
\begin{lemma}
	\label{thm:A}
	For any $(u,v)\in \cH^s$  with $s$ large enough, one has 
	\begin{equation}\label{eq:PL}
		dL(u,v)[(\Xu,\Xv)]=[\Bt(\kappa),L].
	\end{equation}
	where $\Bt$ is a skew-adjoint finite rank operator defined as
	\begin{align}
		\Bt(\kappa)&:=2\alpha U\,\du\wedge\dv
		-2\kappa\beta\,\du\wedge\Rk\dv
		-2\kappa^2\,\Rk\du\wedge\Rk\dv, \label{eq:Bt}
	\end{align}
\end{lemma}
\begin{remark}
  \label{lie.derivative}
The l.h.s. of Eq. \eqref{eq:PL} is the Lie derivative of $L$ as a
function of $(u,v)$ with respect to the Hamiltonian vector field
\eqref{Xu} of the perturbative determinant $\dis$, thus one
immediately gets that along the solutions of the vector field
\eqref{Xu}, the standard Lax equations are fulfilled, namely that one
has
$$
\dot L=[\Bt(\kappa),L].
$$
{However if $\kappa$ is not real then $\Bt(\kappa)$ in general is not
skew-adjoint, therefore the spectrum of $L$ is preserved only for real
values of $\kappa$.}
\end{remark}
\begin{proof}

The proof of \eqref{eq:PL}
  is essentially a direct computation of the Lie derivative of $L= S
  L_1S$ (see \eqref{glg1}).	For every $\xi\in H^1$, by definition
  \begin{align}
\nonumber
		dS(u)[\xi]& = - d\alpha(u)[\xi] \ketbra{\du}{\du} -\alpha(u) (D\xi)\vee \du
		\\
                \nonumber
		&= \alpha(u)(\alpha(u)+A(u)) \braket{\du}{D \xi} \ketbra{\du}{\du}-\alpha(u) (D\xi)\vee \du
	        \\
                \nonumber
			&= \alpha(D\xi\wedge\du)\,S-A\,S\ket{D\xi}\bra{\du}
		\\
                \label{eq:dPb}
			&=-S\,\alpha(D\xi\wedge\du)-A\ket{\du}\bra{D \xi}S
  \end{align}
The equality in the last line arises from taking the adjoint on the
previous line and recalling that $S$ is self-adjoint while
$D\xi\wedge\du$ is skew-adjoint.
	Evaluating \eqref{eq:dPb} at $\xi= \Xu$ we obtain 
		\[
	dS[\Xu]=\alpha\,(D\Xu\wedge\du)\,S-A\,S\ket{D\Xu}\bra{\du},
	\]
	We now prove that
	\begin{equation}\label{eq:BPbis}
		dS[\Xu]=\Bt S-SB_0=-S\Bt-B_0^{*}S .
	\end{equation}
	where
	\begin{equation}
		B_0(\kappa):=-2\kappa\ket{\du}\bra{\Rk\dv}-2AU\ket{\dv}\bra{\du}
		-2\kappa^2\,\Rk\du\wedge\Rk\dv\,, \label{eq:B0}
	\end{equation}
Actually, we only prove the first equality of \eqref{eq:BPbis}, the second one follows from the fact that $S$ is self-adjoint and so is $dS[\Xu]$, while $\tilde B$ is anti self-adjoint.\\
	We start by noting that   proving the identity  \eqref{eq:BPbis} is equivalent to proving that  
		\begin{equation*}
			S N = P S,\quad\mbox{with}\quad
			N:=B_0-A\ket{D\Xu}\bra{\du},\quad 
			P:=\Bt-\alpha\,D\Xu\wedge\du .
		\end{equation*}
			Since $D\Rk v^{(2)}=\dv+\kappa\Rk\dv$, and analogously for $u$, \eqref{Xu} gives
		\begin{equation}\label{eq:DXu}
			D\Xu=-2U\,\dv-2\kappa U\,\Rk\dv-2\kappa Q\,\du-2\kappa^2Q\,\Rk\du .
		\end{equation}
	Thus
		\[
		\begin{aligned}
			P=&-2\kappa\Bigl[(\beta+\alpha U)\,\du\wedge\Rk\dv+\kappa\,\Rk\du\wedge\Rk\dv+\alpha\kappa Q\,\du\wedge\Rk\du\Bigr],\\
			N=&-2\kappa\Bigl[\ket{\du}\bra{\Rk\dv}-AU\ket{\Rk\dv}\bra{\du}-AQ\ket{\du}\bra{\du}\\
			&-\kappa AQ\ket{\Rk\du}\bra{\du}+\kappa\,\Rk\du\wedge\Rk\dv\Bigr].
		\end{aligned}
		\]

		Recalling that 
			\[
		S\Rk\dv=\Rk\dv-\alpha Q\,\du ,\qquad S\Rk\du=\Rk\du-\alpha\braket{\du}{\Rk\du}\,\du ,
		\]  we get
%
	\[
	\begin{aligned}
		&SN-PS=\\ &-2\kappa\Bigl[\;c_1\ket{\du}\bra{\Rk\dv}+c_2\ket{\Rk\du}\bra{\du}+c_3\ket{\Rk\dv}\bra{\du}
		+c_4\ket{\du}\bra{\du}\Bigr],
	\end{aligned}
	\]
	where 
	\[
	\begin{aligned}
		c_1&=\sqrt A-\alpha\kappa\braket{\du}{\Rk\du}-(\beta+\alpha U),\qquad
		c_2=\kappa Q\bigl(\alpha(1+\sqrt A)-A\bigr),\\
			c_3&=-AU-\alpha\kappa\braket{\du}{\Rk\du}+\sqrt A\,(\beta+\alpha U),\\
		c_4&=AQ(\alpha-\sqrt A)+\alpha Q\bigl(\beta+\alpha U+\alpha\kappa\braket{\du}{\Rk\du}\bigr).
	\end{aligned}
	\]
	
Direct algebraic manipulations following from \eqref{varie},
\eqref{eq:def-pq} and using $\beta= \dfrac{\sqrt{A}}{1+\sqrt{A}}$, show that $c_1=c_2=c_3=c_4=0$.
		This proves  the first equality in \eqref{eq:BPbis}.

	Similarly, using
	$d(A^{-1})[\xi]=2\langle\du,D \xi\rangle$ we have
	\begin{equation}\label{eq:dN}
		dL_1[(\xi,\eta)]=- 2  \langle\du,D\xi\rangle\,\ketbra{\dv}{\dv}- A^{-1}\,\bsym{D\eta}{\dv}\,.
	\end{equation}
	As before we compute the differential at $(\Xu,\Xv)$:
	\[
	dL_1[(\Xu,\Xv)]=\bsym{\dv}{\Bigl[-A^{-1}D \Xv-\langle\du,D \Xu\rangle\dv\Bigr]}.
	\]
Inserting
	\[
	\begin{aligned}
		-A^{-1}D\Xv&-\langle\du,D\Xu\rangle\dv\\
		&	=-2AU\,\duuu-\frac{2\kappa V}{A}\,(\du+\kappa\Rk\du)
		-\frac{2\kappa^2Q}{A}\,\Rk\dv+2U\mathtt q_2\,\dv ,
	\end{aligned}
	\]
	and grouping the terms together, we obtain
	\begin{equation}\label{eq:Ncong}
		dL_1[(\Xu,\Xv)]=B_0L_1+L_1B_0^{*}.
	\end{equation}

        Recalling \eqref{eq:definition-L11} and \eqref{glg1} we have that 
	\[
	dL[(\Xu,\Xv)]=dS[\Xu]\,L_1 S+ S\,dL_1[(\Xu,\Xv)]\,S+SL_1\,dS[\Xu].
	\]
	Substituting
	\[
	\begin{aligned}
		dL[\Xu,\Xv]&=(\Bt S-SB_0)L_1S+S(B_0L_1+L_1B_0^{*})S-SL_1(P\Bt+B_0^{*}S)
		\\
		&=\Bt SL_1S- SL_1S\Bt .
	\end{aligned}
	\]
	The four terms containing $B_0$ cancel in pairs, and the right hand side is
	$[\Bt,L]$.
\end{proof} 
\begin{remark}
		Let us note  that, substituting $\kappa=0$ in
                \eqref{eq:Bt} and using \eqref{identitaU}, we recover
                the operator $B$ of the KiPo flow, namely \eqref{eq:LB} (apart from a trivial factor $-2$).
	\end{remark}
	\begin{proof}[Proof of Theorem \ref{BKAPPA}]
		For $N\in\N$ let us consider the set $S_N:=\{\nu\in\sigma(-\Delta_g):\,\nu\leq N \}$. \\
		We denote by $(u_N,\,v_N)$ the projection of $(u,\,v)\in\mathcal H^{1}$ on $V_{S_N}\oplus V_{S_N} $ and by $L_{S_N},\,\Bt_N(\cdot), \dis_N(\cdot)$ 
		the functions $L$, $\Bt(\cdot),\dis(\cdot)$, evaluated at $(u_N,\,v_N)$.\\
		By construction, for any $(u_N,v_N)$ the functions  $\lambda\mapsto \Bt_N(\lambda^{-1})$ is  holomorphic in a neighborhood of $\lambda=0$.  Thus 
		\[
		\Bt_N(\lambda^{-1}) =\sum_{k\ge 2} \lambda^{k-1} \Bt_N^{(k)} \,,\qquad  \Bt_N^{(k)}:= \frac{1}{(k-1)!} \frac{d^{k-1}}{d\lambda^k} \Bt_N(\lambda^{-1})\vert_{\lambda=0}\,.
		\]
		One can readily verify that $B_N^{(k)}= 2 B^{(k)}(u_N,v_N)$
		where $B^{(k)}$, defined in \eqref{Bm},
		is well defined and analytic for $(u,v)\in \cH^s$ with $s$ large enough.
		\\
		Similarly $\dis_N(\lambda^{-1})$ is  holomorphic in a neighborhood of $\lambda=0$ and
		\[
		\dis_N(\cdot)= 1+2\sum_{k\geq 2} \lambda^{k-1} H_{k,N} \,,\qquad \mbox{where}\quad  H_{k,N} = H_k(u_N,v_N)\,.
		\]
Furthermore, since $V_{S_N}\times V_{S_N}$ is invariant for the flow
of $\Delta$
		\[
		(\Xu,\Xv)(u_N,v_N) = 2\sum_{k\geq 2} \lambda^{k-1} X_{H_{k,N} } ,\quad \mbox{with}\quad  X_{H_{k,N} } (u_N,v_N) = X_{H_k}(u_N,v_N)\,.
		\]
In conclusion both sides of \eqref{eq:PL} restricted to $V_{S_N}\times
V_{S_N}$ are holomorphic in a neighborhood of zero and one has
		\begin{equation*}
 2 dL_{S_N}[\sum_{k\geq 2} \lambda^{k-1} X_{H_{k,N} }]=[\sum_{k\ge 2} \lambda^{k-1} \Bt_N^{(k)},L_{S_N}]
		\end{equation*}
So that for any $(u,v)\in V_{S_N}\oplus V_{S_N}$
		\[
		dL(u,v)[X_{H_k}(u,v)] = [B^{(k)},L] (u,v)
		\]
		Finally, since $dL[X_{H_k}],\,B^{(k)}$ are analytic operator valued maps for $(u,v)\in \mathcal H^{s}$ with $s$ sufficiently large , we can conclude by taking the limit for $N\to\infty$ that
		\[
		dL(u,v)[X_{H_k}(u,v)] = [B^{(k)},L] (u,v)\,,\qquad (u,v)\in \mathcal H^{s}\,.
		\]
	\end{proof}

\section{Poisson commutation of the relevant eigenvalues and of the
  functions $H_k$.}\label{commutano}

The main result of this section is a new proof of the fact that the
functions \eqref{def I}  are pairwise Poisson commuting.
This will follow from a Lemma on the Poisson commutativity of the
relevant eigenvalues, that we are now going to state and prove.
First, for a finite $N$ consider (as in the proof of Lemma
\ref{finite gap}) the set $S_N:=\sigma'(-\dg)\cap [0,N]$ and the set
$\cU_{S_N}$ on which the relevant spectrum $\sigma'(L_{S_N})$ is
simple. Denote {$N':=\sum_{\nu\in S_N}\min(n_\nu,2)$ }then, according to
Lemma \ref{finite gap} each relevant eigenvalue
$\left\{\lambda'_n(u,v)\right\}_{n=1}^{N'}$ is simple and also
an analytic function of $(u,v)$.

We have the following Lemma. 

\begin{lemma}
  \label{commutiamo}
For all $j,k\leq N'$, let $\lambda'_{j},\lambda'_k\in\sigma'(L_{S_N})$
be relevant eigenvalues, then one has
\begin{equation}
  \label{davvero}
\left\{\lambda'_{j};\lambda'_{k} \right\}=0\ ,\quad \forall (u,v)\in
\cU_{S_N}. 
  \end{equation}
\end{lemma}
\proof We recall that  $\lambda'_{j};\lambda'_{k}$ are analytic  in $\cU_{S_N}$ and constant along the flow of $\dis(\kappa)$ by Theorem \ref{thm:A}. Hence  one has
$$
\left\{\lambda'_k,\Delta(\kappa)\right\}=0 \ ,\quad \forall (u,v)\in
\cU_{S_N}\ ,\quad \forall \kappa\not\in\sigma'(-\dg)\cap[0,N].
$$
Taking into account equation \eqref{rilevanti}, one has 
\begin{align}
  \label{}
0=&\left\{\lambda'_k,\Delta(\kappa)\right\}= \sum_{j}
\left\{\lambda'_k,\lambda'_j \right\}
\frac{1}{\nu'_j-\kappa}\prod_{m\not=j}\frac{\lambda'_m-\kappa}{\nu'_m-\kappa}\\
= &\left[\prod_{i\in S} \frac{1}{\nu'_i-\kappa}\right] \sum_j {	\left\{\lambda'_k,\lambda'_j \right\}}  \prod_{m\not=j}{(\lambda'_m-\kappa)}
\end{align}
Ignoring the denominator and evaluating this equality at $\kappa=\lambda'_i$,   for $i=1,\dots,N'$, one immediately gets
\[
{	\left\{\lambda'_k,\lambda'_i \right\}}  \prod_{m\not=i}{(\lambda'_m-\lambda'_i)}=0
\]
hence the thesis. \qed

\noindent {\it Proof of Theorem \ref{ridimostro}} We first proof Eq.\eqref{ridi} for $(u,v)\in \cU_{S_N}$. Here
we have that, since $\Delta$ is analytic at infinity, the functions
$H_k$ are actually the coefficients of the expansion of $\Delta$ at
infinity, so that they depend on the variables $(u,v)$ only through
the relevant eigenvalues. With a slight abise of notation we write
$H_{k}(u,v)=H_k(\lambda'(u,v))$. Thus we have
$$
\left\{H_{k_1},H_{k_2}\right\}=\sum_{l_1,l_2}\frac{\partial
  H_{k_1}}{\partial\lambda'_{l_1}} \frac{\partial
  H_{k_2}}{\partial\lambda'_{l_2}}
\left\{\lambda'_{l_1},\lambda'_{l_2}\right\}=0,\quad\forall (u,v)\in
S_N. 
$$
But the l.h.s. is analytic and therefore, by density one immediately
extends such an equality to $V_{S_N}$ and by the density of $V_{S_N}$
one gets the result on the whole space where the both $H_{k_1}$ and
$H_{k_2} $ are analytic. This gives the thesis. \qed

\appendix 

\section{Perturbative determinant and multiplicity of the spectrum}\label{kato}

Let $L_0$ be a closed operator on a Hilbert space $\cH$, and let $P$
be a finite rank perturbation with range contained in the domain of
$L_0$. For any $\kappa\in\rho(L_0)$ (the resolvent of $L_0$), one is interested in the
spectrum of $L:=L_0+P$. We define the perturbative
determinant by
\begin{equation}
\label{WA}
\omega(\kappa,L_0,P):=det(I+PR_\kappa)\ ,
\end{equation}
where $R_\kappa:=(L_0-\kappa)^{-1}$ is the resolvent of
$L_0$ and, according to Eq. III (4.11) of \cite{kato} the determinant
is defined as the determinant of the matrix obtained by restricting
$I+PR_\kappa$ to the range of $P$. Note that, since $P$ is finite rank, this is also the Fredholm determinant of $I+PR_\kappa$.\\

Following \cite{kato} we define the
multiplicity function $\tilde \imath (\kappa, L)$ of 
$L$ by
\begin{equation}
	\label{multi}
	\tti(\kappa,L):=
	\left\{
	\begin{matrix}
		0 & \text{if} & \kappa\in\rho(L)
		\\
		n & \text{if} & \kappa\in\sigma_{pp}(L)\ \text{has\ multiplicity\ }n
		\\
		\infty & \relax & \text{otherwise}
	\end{matrix}
	\right.
\end{equation}
and, given a meromorfic function $f$ on a domain $\cD$, we define a
multiplicity function $\ti(\kappa,f) $ by 
\begin{equation}
	\label{multi.1}
	\ti(\kappa,f):=
	\left\{ 
	\begin{matrix}
		n & \text{if} & \kappa\ \text{is\ a\ zero\ of \ order}\ n
		\\
		-    n & \text{if} & \kappa\ \text{is\ a\ pole \ of \ order}\ n
		\\
		0 & \relax & \text{otherwise}
	\end{matrix}
	\right.
\end{equation}
then one has
the following theorem.

\begin{theorem}
	\label{theo.kato}[Theorem 6.2, p. 247 \cite{kato}]
	Let $L_0$ and $P$ be as above, let $\omega(\kappa)=
	\omega(\kappa,L_0,P)$ be the associated perturbative determinant. If $\cD$
	is a domain of the complex plane consisting of points of $\rho(L_0)$ 
	and of isolated eigenvalues of $L_0$ with finite multiplicities, then $\omega(\kappa,L_0,P)$ is meromorphic
	in $\cD$ and we have for $L=L_0+P$
	\begin{equation}
		\label{WA.A1}
		\tti(\kappa,L)=\tti(\kappa,L_0)+\ti(\kappa,\omega(\cdot,L,P))\ .
	\end{equation}
\end{theorem}
This theorem  shows that the function $\omega$ essentially contains all the
informations on the relation between the spectrum of $L_0$ and the spectrum of $L_0+P$.  

We conclude with a representation formula for $\omega$.
If we assume that $L_0$ has compact resolvent, then we have  a spectral representation of the function $\omega$, which
follows from Krein's theory of perturbative determinants and of spectral
shift functions \cite{Krein53,Krein83} (see also~\cite{BirmanYafaev93}).\\
Let the eigenvalues of \(L_0\) and \(L_0+P\), repeated according to
multiplicity and arranged in nondecreasing order, be respectively
\begin{equation}
\nu_0\leq\nu_1\leq\cdots,
\qquad
\lambda_0\leq\lambda_1\leq\cdots.
\end{equation}

\begin{proposition}
	\label{spectral shift}
{Assume $L_0$ selfadjoint and bounded from below. Assume that also $P$
is selfadjoint.}	For \(\kappa\in\rho(L_0)\),
	one has
	\begin{equation}
\omega(\kappa,L_0,P)
		=\prod_{n=0}^{\infty}
		\frac{\kappa-\lambda_n}{\kappa-\nu_n}.
	\label{product-formula}
	\end{equation}
	where the product 
	converges locally uniformly away from the poles.
\end{proposition}
\proof According to Proposition 2.2 of \cite{yafaev}, the spectral
shift function in the case of operators with compact resolvent can be
defined by 
\begin{equation}
\xi(t)=N_{L_0}(t)-N_{L_0+P}(t),
\label{eq:spectral-shift-counting}
\end{equation}
where
\begin{equation*}
N_T(t):=\Tr\boldsymbol{1}_{(-\infty,t]}(T)
\end{equation*}
then the formula (cf. Eq. (2.4) of \cite{yafaev}) reads
\begin{equation}
\label{lndelta}
\ln\omega(\kappa) =\int_{-\infty}^\infty\frac{\xi(t)}{t-\kappa}dt . 
\end{equation}
Defining
\begin{equation*}
h_n(t):=\mathbf{1}_{[\nu_n,\infty)}(t)
-\mathbf{1}_{[\lambda_n,\infty)}(t),
\end{equation*}
on the one hand one has 
\begin{equation*}
\xi(t)=\sum_{n=0}^{\infty}h_n(t),
\end{equation*}
while on the other 
\begin{equation*}
\int_{\R}\frac{h_n(t)}{t-\kappa}\dd t
=
\int_{\R}
\frac{\boldsymbol{1}_{[\nu_n,\infty)}(t)
	-\boldsymbol{1}_{[\lambda_n,\infty)}(t)}
{t-\kappa}\dd t=\log\frac{\lambda_n-\kappa}{\nu_n-\kappa}.
\end{equation*}
Consequently,
\begin{align}
\int_{\R}\frac{\xi(t)}{t-\kappa}\dd t
=\sum_{n=0}^{\infty}
\log\frac{\lambda_n-\kappa}{\nu_n-\kappa},
\label{eq:integrated-counting-functions}
\end{align}
provided the series is convergent. The convergence of the series is
classical (see e.g. Theorem II of \cite{kato}). \qed

\section{Determinant of finite rank perturbations of the identity}\label{app:rank-one}

{In this section we recall a formula on the determinant of finite rank
perturbations of the identity. A comprehensive theory can be found in \cite{Simon2005,GohbergGoldbergKrupnik2000,GohbergKrein1969,GohbergGoldbergKaashoek1990}.}

Let $\cH$ be a Hilbert space and let $f_1,...,f_k\in\cH$ be indepndent
vectors. Let also be $g_1,...,g_k\in\cH$ be vectors. Consider the
operator
\begin{equation}
  \label{finite.rank}
H:=\Id+\sum_{i=1}^k\ketbra{f_i}{g_i}.
\end{equation}
define the matrix $A$ with matrix elements 
\begin{equation}
  \label{mat.A}
  A_{ij}:=\delta_{i,j}+\langle g_i, f_i\rangle. 
 \end{equation}
\begin{lemma}
  \label{rinite}
  With the above notations
  \begin{equation}
    \label{A.1}
\det(H)=\det(A).
  \end{equation}
\end{lemma}
\proof Consider first the case where the vectors $f_j$, $g_j$ are all
independent, denote by $V$ the linear space generated by
$f_i,...,f_k,g_1,...,g_k$. Then both $V$ and $V^\perp$ are invariant
under $H$. Furthermore $H\big|_{V^\perp}=\Id$, so its determinant
gives just a contribution 1 to the complete determinant. We compute
the determinant of $H\big|_ V$. To this end we use the basis given by
the vectors $f_i,g_i$. Remark that one has
\begin{align}
  \label{basis}
  Hf_j&=f_j+\sum_{i=1}^k\langle g_i,f_j\rangle f_i,
  \\
  \label{basis.1}
  Hg_j&=g_j+\sum_{i=1}^k\langle g_i,g_j\rangle f_i,
  \end{align}
which give the decomposition of $Hf_j$ and $Hg_j$ on the considered
basis. As a consequence the matrix of $H\big|_V$ in such a basis has
the structure
$$
\begin{pmatrix} A & *
  \\
  0&\Id
  \end{pmatrix},
$$
from which the result immediately follows. In the case where the
vectors $g_j$ depend on the $f$'s or are dependent among them, just
keep in the definition of $V$ only those which are independent. \qed

	\bibliographystyle{alpha}
	\bibliography{BibliografiaTotsimone}

\begin{thebibliography}{BHMP26}

\bibitem[BHMP26]{BHMP}
D.~Bambusi, E.~Haus, Simone Marrocco, and M.~Procesi.
\newblock On the integrability of the {K}irchhoff-{P}ohozaev equation on tori.
\newblock {\em arXiv:2609.01891}, preprint 2026.

\bibitem[BM25]{BoitiManfrin2025}
C.~Boiti and R.~Manfrin.
\newblock A third-order conservation law for the {K}irchhoff-{P}okhozhaev
  equation.
\newblock {\em arXiv:2512.12380}, preprint 2025.

\bibitem[BM26]{BoitiManfrin2026}
C.~Boiti and R.~Manfrin.
\newblock A {K}irchhoff equation with infinite conservation laws.
\newblock {\em Nonlinear Analysis}, 273:114219, 2026.

\bibitem[BY93]{BirmanYafaev93}
M.~Sh. Birman and D.~R. Yafaev.
\newblock The spectral shift function. {T}he work of {M}. {G}. {K}rein and its
  further development.
\newblock {\em St. Petersburg Math. J.}, 4(5):833--870, 1993.

\bibitem[CKV26]{campos}
L.~Campos, R.~Killip, and M.~Visan.
\newblock {E}xistence and equicontinuity of solutions to the
  {K}irchhoff--{P}ohozaev wave equation.
\newblock {\em arXiv:2608.25047}, preprint 2026.

\bibitem[GGK90]{GohbergGoldbergKaashoek1990}
I.~Gohberg, S.~Goldberg, and M.~A. Kaashoek.
\newblock {\em Classes of Linear Operators. {V}ol. {I}}, volume~49 of {\em
  Operator Theory: Advances and Applications}.
\newblock Birkh{\"a}user Verlag, Basel, 1990.

\bibitem[GGK00]{GohbergGoldbergKrupnik2000}
I.~Gohberg, S.~Goldberg, and N.~Krupnik.
\newblock {\em Traces and Determinants of Linear Operators}, volume 116 of {\em
  Operator Theory: Advances and Applications}.
\newblock Birkh{\"a}user Verlag, Basel, 2000.

\bibitem[GK69]{GohbergKrein1969}
I.~C. Gohberg and M.~G. Kre\u{\i}n.
\newblock {\em Introduction to the Theory of Linear Nonselfadjoint Operators},
  volume~18 of {\em Translations of Mathematical Monographs}.
\newblock American Mathematical Society, Providence, RI, 1969.
\newblock Translated from the Russian by A. Feinstein.

\bibitem[GK14]{NLS}
B.~Gr\'ebert and T.~Kappeler.
\newblock {\em The defocusing {NLS} equation and its normal form}.
\newblock EMS Series of Lectures in Mathematics. European Mathematical Society
  (EMS), Z\"urich, 2014.

\bibitem[Kat95]{kato}
T.~Kato.
\newblock {\em {P}erturbation Theory for {L}inear {O}perators}.
\newblock Classics in Mathematics. Springer-Verlag, Berlin, Heidelberg, 2nd
  edition, 1995.

\bibitem[KP03]{KAMeKDV}
T.~Kappeler and J.~P\"oschel.
\newblock {\em Kd{V} \& {KAM}}, volume~45 of {\em Ergebnisse der Mathematik und
  ihrer Grenzgebiete. 3. Folge. A Series of Modern Surveys in Mathematics
  [Results in Mathematics and Related Areas. 3rd Series. A Series of Modern
  Surveys in Mathematics]}.
\newblock Springer-Verlag, Berlin, 2003.

\bibitem[Kre53]{Krein53}
M.~G. Krein.
\newblock On the trace formula in perturbation theory.
\newblock {\em Mat. Sbornik N.S.}, 33(75)(3):597--626, 1953.
\newblock (in Russian).

\bibitem[Kre83]{Krein83}
M.~G. Krein.
\newblock On certain new studies in the perturbation theory for selfadjoint
  operators.
\newblock {\em Operator theory}, 7:107--172, 1983.

\bibitem[Poh85]{P2}
S.~I. Poho{\v z}aev.
\newblock A quasilinear hyperbolic kirchhoff equation.
\newblock {\em Differentsial'nye Uravneniya}, 21(1):101--108, 1985.

\bibitem[Sim05]{Simon2005}
B.~Simon.
\newblock {\em Trace Ideals and Their Applications}, volume 120 of {\em
  Mathematical Surveys and Monographs}.
\newblock American Mathematical Society, Providence, RI, second edition, 2005.

\bibitem[Yaf07]{yafaev}
D.~R. Yafaev.
\newblock The schr{\"o}dinger operator: Perturbation determinants, the spectral
  shift function, trace identities, and all that.
\newblock {\em Functional Analysis and Its Applications}, 41:217--236, 2007.

\end{thebibliography}
\end{document}